\documentclass[11pt,a4paper]{article}
\usepackage{amsmath}
\usepackage{amssymb}

\usepackage{eufrak}
\usepackage{euscript}

\newtheorem{theorem}{Theorem}[section]
\newtheorem{lemma}[theorem]{Lemma}
\newtheorem{proposition}[theorem]{Proposition}

\newcommand{\vb}{\vspace{3mm}}

\newcommand{\OU}{\mbox{\sc ou}}
\newcommand{\ou}{\mathrm{OU}}
\newcommand{\ROU}{\mbox{\sc rou}}
\newcommand{\SDE}{\mbox{\sc sde}}

\newcommand{\DD}{{\rm d}}

\newcommand{\rr}{\mathbb{R}}
\newcommand{\one}{\mathbf{1}}
\newcommand{\pp}{\mathbb{P}}
\newcommand{\dd}{\mathrm{d}}
\newcommand{\ee}{\mathbb{E}}

\newenvironment{proof}[1][\proofname]{\par \normalfont \trivlist
 \item[\hskip\labelsep\itshape #1]\ignorespaces
}{%
 \hspace*{\fill}$\Box$ \endtrivlist
}
\newcommand{\proofname}{{\bf Proof}}

\begin{document}
\title{{A note on the central limit theorem for a \\ one-sided reflected Ornstein-Uhlenbeck process}}

\author{Michel Mandjes \& Peter Spreij}
\maketitle
\begin{abstract} \noindent
In this short communication we present a (functional) central limit theorem for the idle process of a one-sided reflected Ornstein-Uhlenbeck proces.
\vb


\vb

\vb

\noindent {\it Affiliations.}
The authors are with Korteweg-de Vries Institute for Mathematics, University of Amsterdam, Science Park 904, 1098 XH  Amsterdam,
the Netherlands; the second author is also with the Radboud University Nijmegen.

\vb

\noindent {\it Email}. $\{$\tt{m.r.h.mandjes,p.j.c.spreij}$\}$\tt{@uva.nl}.
\end{abstract}

\newpage

\section{Introduction}

Ornstein-Uhlenbeck ($\OU$) processes are Markovian, mean reverting Gaussian processes and have found wide-spread use  in a broad range of application domains, such as finance, life sciences, and operations research. In many situations, though, the stochastic process involved is not allowed to cross a certain boundary, or is even supposed to remain within two boundaries. The resulting reflected Ornstein-Uhlenbeck (denoted in the sequel by $\ROU$) processes have been studied by e.g.\ Ward and Glynn \cite{MR1957808,MR1993278,MR2172907}, where  $\ROU$ processes 
are used to approximate the number-in-system processes in M/M/1 and GI/GI/1 queues with reneging under a specific, reasonable scaling. 
Srikant and Whitt \cite{Srikant:1996:SRL:229493.229496} also show that the number-in-system process in a GI/M/$n$ loss model can be approximated by $\ROU$. For other applications, we refer to e.g.\ the introduction of \cite{MR2944002} and references therein.

This note is to be considered as a follow up of, and complementary to, our earlier work~\cite{triple}. That paper considered large deviations results for both one-sided and doubly reflected processes, but only central limit theorems for the `idleness' and `loss' processes in the doubly reflected case. The central limit theorems for one-sided reflected $\OU$ processes are provided in the present note.

\vb
Throughout this note, a probability space $(\Omega, \mathcal{F}, \mathbb{P} )$ equipped with a filtration $\mathbf{F}=\{\mathcal{F}_{t}\}_{t\in \rr_{+}}$ is fixed.
As known, see \cite{LionsSnzitman}, the $\OU$ process is defined
as the unique strong solution to the stochastic differential equation ($\SDE$):
\[
\DD X_{t}=(\alpha-\gamma X_{t})\DD t+\sigma \DD W_{t} , \:\:\:\:X_0=x \in \rr,
\]
where $\alpha\in\rr$, $\gamma, \sigma >0$ and $W_{t}$ is a standard Brownian motion. 

The $\OU$ process is {\it mean-reverting} towards the value $\alpha/\gamma$. To incorporate reflection at a lower boundary $0$, thus constructing $\ROU$, the following $\SDE$ is used, 
\begin{equation}\label{eq:rou}
\DD Y_{t}=(\alpha-\gamma Y_{t})\DD t+\sigma \DD W_{t}+\DD L_{t}, \:\:\:\:Y_{0}=x\geq  0.
\end{equation}
Here $L=\{L_t,t\geq 0\}$ could be interpreted as an `idleness process'. More precisely,  $L$ is defined as the minimal nondecreasing process such that $Y_{t}\geq 0$ for $t\geq 0$; as in the deterministic Skorokhod problem, it holds that $\int_{[0, \infty]}\one_{\{Y_{t}>0\}}\DD L_{t}=0$. Hence for any (continuous) function $g$, one has
\begin{equation}\label{eq:g}
\int_{[0, T]}g(Y_{t})\DD L_{t}=g(0)\int_{[0, T]}\one_{\{Y_{t}=0\}}\DD L_{t}=g(0)L_T, \mbox{ for any $T>0$}.
\end{equation}
Existence of a strong solution to \eqref{eq:rou} has been established in for instance \cite{MR2172907}.

\vb

The  paper mainly focuses on central limit theorems for the idleness process $L$. As in \cite{triple} we use Zhang and Glynn's martingale approach, as developed in \cite{MR2771195} to establish the results. In Section~\ref{SEC5} we review a previous result from \cite{triple} for doubly reflected processes and explain why one has to modify this approach for the one-sided reflected case, whereas in Section~\ref{new} we show which modifications are needed to identify the central limit theorems. We also present results for reflected processes at lower boundaries other than zero, and for processes reflected at an upper bound.

\section{A previous result}\label{SEC5}

Let us briefly summarize the result in \cite{triple}. In that paper the object of study was a doubly reflected (at the lower bound zero, and an upper bound $d
$) $\OU$ process $Z$, satisfying the $\SDE$
\[
\DD Z_t=(\alpha-\gamma Z_t)\DD t+\sigma \DD W_{t}+\DD L_{t}-\DD U_{t}.
\]
For a twice continuously differentiable function $h$  on $\rr$, by It\^{o}'s formula, we have:
\[
\DD h(Z_t)=\big((\alpha-\gamma Z_t)h'(Z_t)+\frac{\sigma^2}{2}h''(Z_t)\big)\DD t+\sigma h'(Z_t)\DD W_{t}+h'(Z_t)\DD L_{t}-h'(Z_t)\DD U_{t}.
\]
Based on the key properties of  $L$ (e.g.\ \eqref{eq:g}) and $U$ (which takes care of the reflection at the upper level $d$), this reduces to 
\begin{align*}
\DD h(Z_t) & =\big((\alpha-\gamma Z_t)h'(Z_t)+\frac{\sigma^2}{2}h''(Z_t)\big)\DD t+\sigma h'(Z_t)\DD W_{t}+h'(0)\DD L_{t}-h'(d)\DD U_{t} \\
& = (\mathcal{L}h)(Y_t)\mathrm{d}t+\sigma h'(Y_t)\DD W_{t}+h'(0)\DD L_{t}-h'(d)\DD U_{t}.
\end{align*}
where the operator ${\mathcal L}$ is defined through
\begin{equation*}
\mathcal{L}:=(\alpha-\gamma x)\frac{\DD}{\DD x}+\frac{\sigma^2}{2}\frac{\DD^2}{\DD x^2}.
\end{equation*}
The following lemma taken from \cite{triple} presented a judicious choice of the function $h$ that was instrumental for the proof of the central limit theorem for the process $U$.
\begin{lemma} \label{lemma7}
Consider the {\sc ode} with real variable right hand side $q\in\rr$
\begin{equation*}
(\mathcal{L}h)=q,\:\:\:\:\:0\leqslant x \leqslant d,
\end{equation*}
under the mixed initial/boundary conditions $h(0)=0,$ $h'(0)=0,$ and $h'(d)=1$.
It has the unique solution $(h,q)\in C^2(\rr)\times\rr$ given by
\[
q  =q_U:=\frac{\sigma^{2}}{2}
\frac{W(d)}{\int_0^d W(v)\DD v}, \:\:\:\:\:
h(x) =\frac{2q_U}{\sigma^{2}}\int_{0}^{x}\int_0^u
\frac{W(v)}{W(u)}\DD v\,\DD u,
\]
where
\begin{equation*}
W(v):=\exp\left(\frac{2\alpha v}{\sigma^2}-\frac{\gamma v^2}{\sigma^2}\right).
\end{equation*}
\end{lemma}
Indeed, with this choice of $h$ we have
\[
\DD h(Z_t)  = \sigma h'(Y_t)\DD W_{t}-(\DD U_{t}-q_U\,\mathrm{d}t).
\]
Boundedness of $h(Z_t)$ ($Z$ is a bounded process) combined with a central limit theorem for the martingale $\int_0^\cdot \sigma h'(Y_t)\DD W_{t}$ was central in the proof of the central limit theorem for $U_t$, see \cite{triple} for further details.  
\medskip\\
In the present paper, we deal with the one-sided reflected process $Y$ and with a twice differential function $h$ one obtains
\begin{equation}\label{HH}
\DD h(Y_t)=(\mathcal{L}h)(Y_t)\mathrm{d}t+\sigma h'(Y_t)\DD W_{t}+h'(0)\DD L_{t}.
\end{equation}
Two facts obstruct a direct application of the method above: (1) the process $h(Y_t)$ is not bounded, and (2) we cannot immediately apply Lemma~\ref{lemma7} to get a proper choice of $h$. Indeed, we needed three initial/boundary conditions to also determine the constant $q$, whereas now, we can only specify $h(0)$ and $h'(0)$. In the next section, we will see how to overcome these difficulties.

\section{The central limit theorems}\label{new}

The main objective of this section is to derive a central limit theorem for $L_t$, for $t\to\infty$, and a functional version of it. We do so relying on the martingale techniques initiated in \cite{MR2771195}. We also consider other reflected processes.

\subsection{Main results}

Dealing with only a one-sided reflected process, we argue that the  procedure as outlined in Section~\ref{SEC5} breaks down. 
In order to remedy this difficulty, we modify the procedure as follows. We need a separate argument that establishes the value of $q_L$ that appears in Theorem~\ref{thm:clt} and Theorem~\ref{thm:fclt} and the following variant of  Lemma~\ref{lemma7}.

\begin{lemma}\label{lemma:h}
Let $q$ be a given constant. Consider the {\sc ode}
\begin{equation*}
(\mathcal{L}h)=q,\, x\geq 0,
\end{equation*}
under the  initial conditions $h(0)=0,$ $h'(0)=1$.
It has the unique solution $h\in C^2(\rr)$ given by
\[
h(x)=\int_{0}^{x}\frac{1}{W(u)}\left(1+
\frac{2q}{\sigma^{2}}\int_0^u
W(v)\DD v\right)\,\DD u,
\]
where
\begin{equation*}
W(v):=\exp\left(\frac{2\alpha v}{\sigma^2}-\frac{\gamma v^2}{\sigma^2}\right).
\end{equation*}
\end{lemma}

\begin{proof} One easily verifies (like in the proof of the corresponding result in \cite{triple}) that the general solution is
\[h(x) =C_{2}+\int_{0}^{x}f(u)\DD u=C_{2}+C_{1}\int_{0}^{x}\frac{1}{W(u)}\DD u+
\frac{2q}{\sigma^{2}}\int_{0}^{x}\int_0^u
\frac{W(v)}{W(u)}\DD v\,\DD u.
\]
Then the initial conditions $h(0)=0, h'(0)=1$ uniquely determine the values of $C_{1}, C_{2}$. Indeed, we get $C_2=0$ and $C_1=1$, and so
\[
h(x)=\int_{0}^{x}\frac{1}{W(u)}\DD u+
\frac{2q}{\sigma^{2}}\int_{0}^{x}\int_0^u
\frac{W(v)}{W(u)}\DD v\,\DD u.
\]
\end{proof}

Next we give the solution as presented in Lemma~\ref{lemma:h} a different appearance. First we need that the fact the invariant distribution of $Y$ is truncated normal, see \cite[Proposition 1]{MR1993278}. If $X$ is the solution to a non-reflected equation, an ordinary $\OU$ process, its invariant distribution is $N(\frac{\alpha}{\gamma},\frac{\sigma^2}{2\gamma})$. Let $\xi$ be a random variable having this distribution and denote its density by $p_{\ou}(x)$. By a simple computation one gets
\[
W(x)=p_{\ou}(x)\exp(\frac{\alpha^2}{\gamma\sigma^2})\sqrt{\pi\sigma^2/\gamma}.
\]
Furthermore, $\pp(\xi>0)=\int_0^\infty p_{\ou}(x)\,\dd x=\Phi(\frac{\alpha}{\sqrt{\sigma^2\gamma/2}})$ with $\Phi$ the cdf of the standard normal distribution. The invariant density $p_Y$ of $Y$ is given by (here and further below $y\geq 0$)
\[
p_Y(y)=\frac{p_{\ou}(y)}{\int_0^\infty p_{\ou}(u)\,\dd u}=\frac{W(y)}{\int_0^\infty W(u)\,\dd u},
\]
or, in explicit terms,
\[
p_Y(y)=\frac{1}{\pp(\xi>0)\sqrt{\pi\sigma^2/\gamma}}\exp\left(-\frac{\gamma}{\sigma^2}(y-\frac{\alpha}{\gamma})^2\right).
\]
Note further that 
\[
p_Y(0)=\frac{p_{\ou}(0)}{\pp(\xi>0)}=\frac{\exp(-\frac{\alpha^2}{\gamma\sigma^2})}{\pp(\xi>0)\sqrt{\pi\sigma^2/\gamma}},
\]
from which it follows that 
\begin{equation}\label{eq:wp}
W(y)=\frac{p_Y(y)}{p_Y(0)}.
\end{equation}

Let $\eta$ be a random variable with density $p$. We proceed by computing $\ee\eta$. 

\begin{lemma}\label{lemma:eeta}
It holds that 
\[
\ee\eta=\int_0^\infty yp_Y(y)\,\dd y = \frac{\sigma^2}{2\gamma}p_Y(0)+\frac{\alpha}{\gamma}.
\]
Let $Y$ be in its stationary regime. Then $q_L:=\frac{\dd \ee L_t}{\dd t}=\frac{\sigma^2}{2}p_Y(0)$.
\end{lemma} 

\begin{proof}
Note the  identity
\[
\frac{\dd p_Y(y)}{\dd y}= -\frac{2\gamma}{\sigma^2}(y-\frac{\alpha}{\gamma})p_Y(y).
\]
Hence,
\begin{align*}
\ee\eta 
& =\int_0^\infty (y-\frac{\alpha}{\gamma})p_Y(y)\,\dd y +\int_0^\infty \frac{\alpha}{\gamma}p_Y(y)\,\dd y \\
& =-\frac{\sigma^2}{2\gamma}\int_0^\infty \frac{\dd p_Y(y)}{\dd y}\,\dd y +\frac{\alpha}{\gamma} \\
& =\frac{\sigma^2}{2\gamma} p_Y(0)+\frac{\alpha}{\gamma}.
\end{align*}

The next step is to determine $q_L$. 
Let $Y$ be in its stationary regime. From the $\SDE$ for $Y$ we get
\[
0=(\alpha-\gamma\, \ee\eta)+q_L.
\]
Using the above expression for $\ee\eta$, we get $
q_L=\gamma\ee\eta-\alpha=\frac{\sigma^2}{2}p_Y(0)$.
\end{proof}


\begin{lemma}\label{lemma:inth}
Let $q=-q_L$ and let $h$ be the function as in Lemma~\ref{lemma:h}. Then, for $x>0$, $h'(x)=\frac{p_Y(0)}{p_Y(x)}\bar F_Y(x)$, where $\bar F_Y(x)=1-F_Y(x)$, with $F_Y$ the distribution function associated to the invariant density $p_Y$. Moreover, $\int_0^\infty h'(x)^2p_Y(x)\,\dd x<\infty$.
\end{lemma} 

\begin{proof}
First we note that 
\[
h(x)=\frac{1}{W(x)}\left(1-
\frac{2q_L}{\sigma^{2}}\int_0^x
W(v)\DD v\right).
\]
Use now $q_L=\frac{\sigma^2}{2}p_Y(0)$ to get $1-
\frac{2q_L}{\sigma^{2}}\int_0^x
W(v)\DD v=1-p_Y(0)\int_0^xW(v)\,\dd v$ and, recall \eqref{eq:wp}, $p_Y(0)W(v)=p_Y(v)$ to arrive at $1-
\frac{2q_L}{\sigma^{2}}\int_0^x
W(v)\DD v=1-\int_0^xp_Y(v)\,\dd v=\bar F_Y(x)$. The first result follows.

To prove the second result, we note that 
$p_Y(x)/\bar F_Y(x)=p_{\ou}(x)/\bar F_{\ou}(x)$.
Recall that for $x\to\infty$ it holds that $\frac{\bar\Phi(x)}{\phi(x)}\sim \frac{1}{x}$ ($\phi$ is the density of $N(0,1)$). Hence, we also have $\frac{\bar F_Y(x)}{p_Y(x)}\sim \frac{1}{x}$. Hence for large $x$ we have
$\int_x^\infty h'(y)^2p_Y(y)\,\dd y <\infty$.
\end{proof}
Here is the first central limit theorem, the counterpart of Proposition 6 in \cite{triple}.
\begin{theorem}\label{thm:clt}
Let $h$ be as in Lemma~\ref{lemma:h} for $q=-q_L$. It holds that $\frac{L_t-q_Lt}{\sqrt{t}}$ weakly converges to $N(0,\tau^2)$, where $\tau^2=\sigma^2\int_0^\infty h'(x)^2p_Y(x)\,\dd x<\infty$.
\end{theorem}

\begin{proof}
Let $h$ be as in Lemma~\ref{lemma:h} for $q=-q_L$, and consider $h(Y_t)$. It\^o's rule gives
\[
h(Y_t)-h(Y_0)= \int_0^t\mathcal{L}h(Y_s)\,\dd s+\sigma\int_0^th'(Y_s)\,\dd W_s+\int_0^th'(Y_s)\,\dd L_s.
\]
Property \eqref{eq:g} of $L$ together with $h'(0)=1$ give $\int_0^th'(Y_s)\,\dd L_s=L_t$. Combine this with $\mathcal{L}h(Y_s)=-q_L$ to arrive at
\[
h(Y_t)-h(Y_0)= -q_Lt+\sigma\int_0^th'(Y_s)\,\dd W_s+L_t.
\]
Hence,
\[
\frac{L_t-q_Lt}{\sqrt{t}}=\frac{h(Y_t)-h(Y_0)}{\sqrt{t}}-\frac{1}{\sqrt{t}}\sigma\int_0^th'(Y_s)\,\dd W_s.
\]
As $Y_t\to \eta$ in distribution, where $\eta$ is distributed according to the invariant distribution of $Y$, we also have by the continuous mapping theorem, $h(Y_t)\to h(\eta)$ in distribution, and hence $\frac{h(Y_t)-h(Y_0)}{\sqrt{t}}\stackrel{\pp}{\to} 0$. By the ergodic theorem~\cite[p.\ 134]{MR0346904},
\[
\frac{1}{t}\langle \sigma\int_0^\cdot h'(Y_s)\,\dd W_s\rangle_t\stackrel{\pp}{\to}\sigma^2\int_0^\infty h'(x)^2p_Y(x)\,\dd x=\tau^2,
\]
where the right hand side is finite according to Lemma~\ref{lemma:inth}. Hence $\frac{1}{\sqrt{t}}\sigma\int_0^th'(Y_s)\,\dd W_s\to N(0,\tau^2)$ in distribution. 
\end{proof}

With a bit more effort, we obtain a functional version of the above theorem.

\begin{theorem}\label{thm:fclt}
The centered and scaled loss process $L^n$ defined by 
$L_{t}^{n}:=\frac{L_{nt}-q_Lnt}{\tau\sqrt{n}}$ converges 
weakly in  $C[0,\infty)$ with the locally uniform metric  to a standard Brownian motion as $n\rightarrow \infty$.
\end{theorem}

\begin{proof}  We have, as in the proof of Theorem~\ref{thm:clt},
\begin{equation*}
\frac{L_{nt}-q_Lnt}{\sqrt{n}}=\frac{h(Y_{nt})-h(Y_{0})}{\sqrt{n}}-\frac{\sigma}{\sqrt{n}} \int_{0}^{nt}h'(Y_{s})\mathrm{d}W_{s}.
\end{equation*}
By the ergodic theorem \cite[p.\ 134]{MR0346904}, for arbitrary $t\in[0, \infty)$,
\begin{equation*} 
\langle \frac{\sigma}{\sqrt{n}} \int_{0}^{n\cdot}h'(Y_{s})\mathrm{d}W_{s} \rangle_{t}=\frac{\sigma^{2}t}{nt} \int_{0}^{nt}h'(Y_{s})^{2}\DD s 
\stackrel{\mathbb{P}}{\rightarrow} \tau^2t,\:\:\text{as } n\rightarrow \infty,
\end{equation*} 
and hence, by the martingale central limit theorem, we have weak convergence of the martingales $\frac{\sigma}{\sqrt{n}} \int_{0}^{n\cdot}h'(Y_{s})\mathrm{d}W_{s}$ to a Brownian motion.
The claim will be proved by applying the functional limit theorem for semimartingales \cite[Thm. 3] {MR651472} to $L^{n}$, for which it is now sufficient to show that for every $T>0$
\[\sup_{t\leq T}\frac{h(Y_{0})-h(Y_{nt})}{\sqrt{n}}\rightarrow 0 \: \:\: \text{in probability for } n\rightarrow \infty.\]
We have seen in the proof of Lemma~\ref{lemma:inth} that $h'(x)$ behaves as  $\frac{1}{x}$ for large values of $x$. Hence $h(x)$ for large $x$ behaves as  $\log x$, and therefore it is sufficient to show that 
\begin{equation}\label{eq:sup}
\sup_{t\leq T}\frac{\log (Y_{nt}+1)}{\sqrt{n}}\rightarrow 0 \:\:\: \text{in probability for } n\rightarrow \infty.
\end{equation}
We use the following trivial estimate. Since $Y_t\geq 0$, we have
$Y_t\leq Y_0+\alpha t +\sigma W_t + L_t$. Hence, denoting $W^*_t=\sup_{s\leq t}|W_s|$ and recalling Doob's inequality, $\ee(W^*_t)^2\leq 4t$, we have $\sup_{t\leq T} Y_{nt}\leq \alpha nT+\sigma W^*_{nT}+L_{nT}$ and $\ee \sup_{t\leq T} Y_{nt}\leq \alpha nT+\sqrt{\ee(\sigma W^*_{nT})^2}+\ee L_{nT}=\alpha_nT+ \sqrt{\sigma^24nT}+\ee L_{nT}$. Hence $\limsup_{n\to\infty}\frac{1}{n}\ee\sup_{t\leq T} Y_{nt}\leq \alpha T+q_LT$. But then, by Jensen's inequality, 
\[
\limsup_{n\to \infty}\frac{1}{n}\ee\sup_{t\leq T}\log (Y_{nt}+1)\leq \log \limsup_{n\to \infty}\frac{1}{n}\ee\sup_{t\leq T} (Y_{nt}+1)<\infty. 
\]
We conclude that  $\ee\sup_{t\leq T}\frac{\log (Y_{nt}+1)}{\sqrt{n}}\to 0$, as $n\to\infty$, which is sufficient for \eqref{eq:sup} to hold, as the supremum is trivially nonnegative.
\end{proof}

\subsection{Other reflecting boundaries}

In this section we study processes that are (lower or upper) reflected at other boundaries. As the results immediately follow from Theorem~\ref{thm:fclt} or can be proven in a similar fashion, we state them without proofs.

First we consider a process reflected at another lower boundary than zero, which we reduce by translation to the previous case. Let $\ell$ be this boundary 
and consider the one-sided lower reflected process $Y^\ell$ that is such that $Y^\ell_t\geq \ell$, and given by the $\SDE$
\[
\dd Y^\ell_t=(\alpha-\gamma Y^\ell_t)\,\dd t +\sigma\dd W_t+\dd L^\ell_t,
\]
where $L^\ell$ is the minimal increasing process that renders $Y^\ell_t\geq \ell$ for all $t\geq 0$. Put $Y_t:=Y^\ell_t-\ell$ to find
\[
\dd  Y_t=(\alpha^\ell-\gamma  Y_t)\,\dd t+\sigma\dd W_t+\dd L^\ell_t,
\]
with $\alpha^\ell=\alpha-\gamma \ell$. Note that $\int_0^\infty \one_{
\{Y_t>0\}}\,\dd L^\ell_t=0$. It follows that one can obtain a central limit theorem for $\L^\ell$ from the result in the previous section. We need that the stationary density of $Y^\ell$ is (at $\ell$) truncated normal. For $y>\ell$ one has
\[
p_{Y^\ell}(y)=\frac{1}{\Phi(\frac{\alpha-\gamma \ell}
{\sqrt{\sigma^2\gamma/2}})}
\frac{1}{\sqrt{\pi\sigma^2/\gamma}}\exp(-\frac{\gamma}{\sigma^2}(y-\frac{\alpha}{\gamma})^2).
\]
We also need $q_{\ell}=\frac{\sigma^2}{2}p_{Y^\ell}(\ell)$ and the function $h_\ell$, which is for $y>\ell$ given by $h_\ell(y)=\int_\ell^yh_\ell'(x)\,\dd x$, with $h_\ell'(x)=\frac{p_{Y^\ell}(\ell)}{p_{Y^\ell}(x)}\bar F_{Y^\ell}(x)$, and $\tau^2_\ell=\sigma^2\int_\ell^\infty (h_\ell'(x))^2p_{Y^\ell}(x)\,\dd x$. Note that $\mathcal{L}h(x)=-q_\ell$.
The precise result is as follows.

\begin{proposition}\label{prop:fcltl}
The centered and scaled loss process $L^{\ell,n}$ defined by 
$L_{t}^{\ell,n}:=\frac{L^\ell_{nt}-q_{\ell}nt}{\tau_\ell\sqrt{n}}$ converges 
weakly in  $C[0,\infty)$ with the locally uniform metric  to a Brownian motion as $n\rightarrow \infty$.
\end{proposition}

Next we turn to upper reflected processes.
Let $d\in\rr$ 
and consider the one-sided upper reflected process $Z$ that is such that $Z_t\leq d$, and given by the $\SDE$
\[
\dd Z_t=(\alpha-\gamma Z_t)\,\dd t +\sigma\dd W_t-\dd U_t,
\]
where $U$ is the minimal increasing process that renders $Z_t\leq d$ for all $t\geq 0$. Note that $\int_0^\infty\one_{\{Z_t< d\}}\,\dd U_t=0$. By `flipping' we can reduce this case to the one with reflection at a lower boundary.
Let $\tilde Y_t:=d-Z_t$, then we find
\[
\dd \tilde Y_t=(\tilde\alpha-\gamma \tilde Y_t)\,\dd t-\sigma\dd W_t+\dd U_t,
\]
with $\tilde\alpha=-\alpha+\gamma d$. It follows that one can obtain a central limit theorem for $U$ from the results in the previous section. Almost all that is needed is to express all quantities needed in terms of $\tilde\alpha=-\alpha+\gamma d$ instead of in $\alpha$. For instance, the invariant density $p_Z$ of $Z$ can be derived from the invariant density of $\tilde Y$. It is at zero truncated $N(\frac{\tilde\alpha}{\gamma},\frac{\sigma^2}{2\gamma})$ and one has (for $z<d$) $p_Z(z)=p_{\tilde Y}(d-z)$, explicitly,
\[
p_Z(z)=\frac{1}{\Phi(\frac{\gamma d-\alpha)}{\sigma^2\gamma/2})}\frac{1}{\sqrt{\frac{\pi\sigma^2}{2\gamma}}}\exp(-\frac{\gamma}{\sigma^2}(z-\frac{\alpha}{\gamma})^2).
\]
We also need (for $z\leq d$) $h_Z(z)=-\int_z^d h_Z'(u)\,\dd u$, where  $h_Z'(z):=\frac{p_Z(d)}{p_Z(z)}F_Z(z)$, with $F_Z(z)=\int_{-\infty}^zp_Z(u)\,\dd u$. Note that $h_Z(d)=0$, $h_Z'(d)=1$ and $\mathcal{L}h_Z(z)=q_U$, where $q_U=\frac{\sigma^2}{2}p_Z(d)$.

The precise result is as follows. 

\begin{proposition}
Let $q_U=\frac{\sigma^2}{2}p_Z(d)$
and $\tau_U^2=\sigma^2\int_{-\infty}^dh_Z'(z)^2p_Z(z)\,\dd z$. For $n\to\infty$ we have weak convergence of the scaled and centered process $U^n$ defined by $U^n_t=(n\tau_U^2)^{-1/2}(U_{nt}-q_Unt)$ to a standard Brownian motion.
\end{proposition}

{\small
\bibliographystyle{plain}

}

\end{document}